\documentclass[a4paper,fleqn]{article}
\usepackage{amssymb,amsthm,amsfonts}
\usepackage{amsmath}
\newtheorem{theorem}{Theorem}

\numberwithin{equation}{section}
\numberwithin{lemma}{section}
\numberwithin{theorem}{section}
\numberwithin{corollary}{section}

\allowdisplaybreaks
\begin{document}
\setcounter{page}{1}

\markboth{V. Sahai and  A. Verma}{Recursion formulas}

\title{On the recursion formulas for the matrix special functions of one and two variables}

\author{Vivek Sahai
\\ 
Department of Mathematics and Astronomy\\ Lucknow University \\Lucknow 226007, India 
\\
sahai\_vivek@hotmail.com
\\[10pt]
Ashish Verma\footnote{Corresponding author}
\\ 
Department of Mathematics\\ Prof. Rajendra Singh (Rajju Bhaiya)\\ Institute of Physical Sciences for Study and Research \\  V. B. S. Purvanchal University, Jaunpur  (U.P)- 222003, India\\
vashish.lu@gmail.com}

\maketitle
\begin{abstract}
Special matrix functions have recently been investigated for regions of convergence, integral representations and the systems of matrix differential equation that these functions satisfy. In this paper, we find the recursion formulas for the Gauss hypergeometric matrix function. We also give the recursion formulas for the two variable Appell matrix functions. \\[12pt]
\end{abstract}

\section{Introduction}
The theory of matrix special functions has attracted considerable attention in the last two decades. Special matrix functions appear in the literature related to statistics \cite{A}, Lie theory \cite{AT} and in connection with the matrix  version of Laguerre, Hermite and Legendre differential equations and the corresponding polynomial families \cite{LR,LRE,LR1}. Recently, Abd-Elmageed \textit{et. al.} \cite{AAAS} have obtained numerous contiguous and recursion formulas satisfied by the first Appell matrix function, namely $F_1$. Motivated by this study, in the present paper, we study recursion formulas for the Gauss hypergeometric matrix function and all four Appell matrix functions. 

Recursion formulas for the Appell functions have been studied in the literature, see Opps, Saad and Srivastava \cite{OP} and Wang \cite{XW}.  Recursion formulas for several multivariable hypergeometric functions were presented in \cite{VS1, VS2, VS3, VS4} . 

The paper is organized as follows. In Section~2, we give a review of basic definitions that are needed in the sequel. In Section~3, we obtain the recursion formulas for the Gauss hypergeometric function. In Section~4, recursion formulas for the four Appell matrix functions are obtained. We remark that our results generalize the corresponding results in \cite{AAAS}, where all the matrices involved are assumed to be commuting.

\section{Preliminaries}
Let $\mathbb{C}^{r\times r}$ be the vector space of $r$-square matrices with complex entries. For any matrix $A\in \mathbb{C}^{r\times r}$, its spectrum $\sigma(A)$ is the set of eigenvalues of $A$. If $f(z)$ and $g(z)$ are holomorphic functions of the complex variable $z$, which are defined in an open set $\Omega$ of the complex plane and $A\in \mathbb{C}^{r\times r}$ with $\sigma(A)\subset\Omega$, then from the properties of the matrix functional calculus \cite{ND}, we have  $f(A) g(A)= g(A) f(A)$. If $B\in\mathbb{C}^{r\times r}$ is a matrix for which $\sigma(B)\subset\Omega$, and if $AB=BA$, then $f(A) g(B)= g(B) f(A)$. A square matrix $A$ in $\mathbb{C}^{r\times r}$  is said to be positive stable if $\Re(\lambda)>0$ for all $\lambda\in\sigma(A)$. 

The reciprocal gamma function $\Gamma^{-1}(z)=1/\Gamma(z)$ is an entire function of the complex variable $z$. The image of $\Gamma^{-1}(z)$ acting on $A$, denoted by $\Gamma^{-1}(A)$, is a well defined matrix. If $A+nI$ is invertible for all integers $n\geq 0$, then the reciprocal gamma function \cite{LC1} is defined as 
$\Gamma^{-1}(A)= (A)_n \ \Gamma^{-1}(A+nI)$, where $(A)_n$ is the shifted factorial matrix function for $A\in\mathbb{C}^{r\times r}$ given by \cite{LC}
\begin{align*} 
	(A)_n=
\begin{cases}
I ,
& n=0,\\
 A(A+I) \cdots (A+(n-1)I) ,
&n\geq 1 .
\end{cases}
\end{align*}
$I$ being the $r$-square identity matrix.
If $A\in\mathbb{C}^{r\times r}$ is a positive stable matrix and $n\geq1$, then by \cite{LC1} we have $\Gamma(A) = \lim_{n \to\infty}(n-1)! (A)^{-1}_{n} n^A$.  For a good survey of theory of special matrix functions, see \cite{MA}.

The Gauss hypergeometric  matrix function \cite{LC} is defined by 
\begin{align}
_2F_1(A, B; C; x)=\sum_{n=0}^{\infty}\frac{(A)_{n}(B)_{n}(C)_{n}^{-1}}{n!}\, x^{n},\label{1eq1}
\end{align}
for matrices $A$, $B$ and $C$ in $\mathbb{C}^{r\times r}$ such that $C+nI$ is invertible for all integers  $n\geq 0$ and $|x|\leq 1$.

The Appell  matrix functions are defined as 
\begin{align}
F_{1}(A, B, B'; C;x, y)&=\sum_{m, n=0}^{\infty}(A)_{m+n}(B)_{m}(B')_{n}(C)_{m+n}^{-1} \, \frac{x^m\, y^n}{m! \, n!},\label{1eq2}
\\
F_{2}(A, B, B'; C, C';x, y)&=\sum_{m, n=0}^{\infty}(A)_{m+n}(B)_{m}(B')_{n}(C)_{m}^{-1}(C')_{n}^{-1} \, \frac{x^m\, y^n}{m! \, n!},\label{1eq3}
\\
F_{3}(A, A', B, B'; C;x, y)&=\sum_{m, n=0}^{\infty}(A)_{m}(A')_{n}(B)_{m}(B')_{n}(C)_{m+n}^{-1}\, \frac{x^m\, y^n}{m! \, n!},\label{1eq4}
\\
F_{4}(A, B; C, C';x, y)&=\sum_{m, n=0}^{\infty}(A)_{m+n}(B)_{m+n}(C)_{m}^{-1}(C')_{n}^{-1}\, \frac{x^m\, y^n}{m! \, n!},\label{1eq5}
\end{align}
where $A$, $A'$, $B$, $B'$, $C$, $C'\in\mathbb{C}^{r\times r}$  such that $C+nI$  and $C'+nI$ are invertible for all  integers $n\geq 0$. For regions of convergence of (\ref{1eq2})-(\ref{1eq5}), see \cite{QM,RD,RD3}.

\section{Recursion formulas for the Gauss hypergeometric matrix function}
In this section, we obtain the recursion formulas for the Gauss hypergeometric matrix function. Throughout the paper, $I$ denotes the identity matrix and $s$ denotes a non-negative integer. 
\begin{theorem}\label{rth12} Let $A+sI$ be  invertible  for all  integers $s\geq0$ and let  $B$, $C$ be commuting matrices. Then the following recursion formula holds  true for the Gauss hypergeometric matrix function $_2F_{1}(A, B; C; x)$:
\begin{align}
&_2F_{1}(A+sI, B; C; x)={_2F_{1}}(A, B; C; x)+ x\left[\sum_{k=1}^{s}{_2F_{1}}(A+kI, B+I; C+I; x)\right]{B}{C}^{-1}.
\label{2eq1}
\end{align}
Furthermore, if $A-kI$ is invertible for  integers $k\leq s$, then
\begin{align}
&_2F_{1}(A-sI, B; C; x)={_2F_{1}}(A, B; C; x)- x\left[\sum_{k=0}^{s-1}{_2F_{1}}(A-kI, B+I; C+I; x)\right]{B}{C}^{-1}.
\label{2eq2}
\end{align}
\end{theorem}
\begin{proof} 
From the definition of the Gauss hypergeometric  matrix function (\ref{1eq1}) and  the relation
\begin{align}
(A+I)_n = A^{-1} (A)_{n}(A+nI) ,\label{2eqp1}
\end{align}
we get the following contiguous relation:
\begin{align}
&_2F_{1}(A+I, B; C; x)={_2F_{1}}(A, B; C; x)+ x\,\Big[{_2F_{1}}(A+I, B+I; C+I; x)\Big]{B}{C}^{-1}.
\label{2eqp2}
\end{align}
Replacing $A$ with $A+I$ in (\ref{2eqp2}) and using (\ref{2eqp2}), we have the following contiguous relation for $_2F_{1}(A+2I, B; C; x)$:
\begin{align}
&_2F_{1}(A+2I, B; C; x)={_2F_{1}}(A, B; C; x)+ x\left[\sum_{k=1}^{2}{_2F_{1}}(A+kI, B+I; C+I; x)\right]{B}{C}^{-1}.
\label{2eqp3}
\end{align}
Iterating this process $s$ times, we get (\ref{2eq1}).  For the proof of (\ref{2eq2}) replace the matrix $A$ with $A-I$ in (\ref{2eqp2}). As $A-I$ is invertible, we have
\begin{align}
&_2F_{1}(A-I, B; C; x)={_2F_{1}}(A, B; C; x)- x \,\Big[{_2F_{1}}(A, B+I; C+I; x)\Big]{B}{C}^{-1}.\label{2eqp4}
\end{align}
Iteratively, we get (\ref{2eq2}).
\end{proof}

Using contiguous relations (\ref{2eqp2}) and (\ref{2eqp4}), we get following forms of the recursion formulas for $_2F_{1}(A, B; C; x)$.
\begin{theorem}\label{rth21}
Let $A+sI$ be an invertible matrix for all   integers $s\geq0$ and let $B$, $C$ be commuting matrices. Then the following recursion formula holds true for the Gauss hypergeometric matrix function $_2F_{1}(A, B; C; x)$: 
\begin{align}
&{_2F_{1}}(A+sI, B; C; x)=\sum_{k=0}^{s}{s\choose k}\, x^{k}\,\Big[{_2F_{1}}(A+kI, B+kI; C+kI; x)\Big]{(B)_{k}}{(C)^{-1}_{k}},
\label{2eq3}
\end{align}
Furthermore, if $A-kI$ is invertible for  integers  $k\leq s$, then
\begin{align}
&{_2F_{1}}(A-sI, B; C; x)=\sum_{k=0}^{s}{s\choose k}\, (-x)^{k}\,\Big[{_2F_{1}}(A, B+kI; C+kI; x)\Big]{(B)_{k}}{(C)^{-1}_{k}}.\label{2eq4}
\end{align}
\end{theorem}
\begin{proof}  
We prove (\ref{2eq3}) by applying mathematical induction on $s$. For $s=1$, the result (\ref{2eq3}) is true by (\ref{2eqp2}). Assuming (\ref{2eq3}) is true for $s=t$, that is, 
\begin{align}
&{_2F_{1}}(A+tI, B; C; x)=\sum_{k=0}^{t}{t\choose k}x^{k}\,\Big[{_2F_{1}}(A+kI, B+kI; C+kI; x)\Big]{(B)_{k}}{(C)^{-1}_{k}}.
\label{2eqpp1}
\end{align}
Replacing $A$ with $A+I$ in (\ref{2eqpp1}) and using the contiguous relation (\ref{2eqp2}), we get 
\begin{align}
&{_2F_{1}}(A+(t+1)I, B; C; x)\notag\\
&=\sum_{k=0}^{t}{t\choose k}\, x^{k}\,\Big[{_2F_{1}}(A+kI, B+kI; C+kI; x)\notag\\
&\quad +x\left({_2F_{1}}(A+(k+1)I, B+(k+1)I; C+(k+1)I; x)\right){(B+kI)}{(C+kI)^{-1}}\Big]{(B)_{k}}{(C)^{-1}_{k}}.\label{2eqpp2}
\end{align}
Simplifying, (\ref{2eqpp2}) takes the form 
\begin{align}
&{_2F_{1}}(A+(t+1)I, B; C; x)\notag\\&=\sum_{k=0}^{t+1}\left[{t\choose k}+{t\choose k-1}\right]\, x^{k}\,\{{_2F_{1}}(A+kI, B+kI; C+kI; x)\}{(B)_{k}}{(C)_{k}^{-1}}.\label{2eqpp3}
\end{align}
Apply the known relation ${n\choose k}+{n\choose k-1}={n+1\choose k}$ and ${n\choose k}=0$\, ($k>n$ or $k<0$), the above identity can be reduced to the following result:
\begin{align}
&{_2F_{1}}(A+(t+1)I, B; C; x)\notag\\
&=\sum_{k=0}^{t+1}{t+1\choose k}\, x^{k}\,\Big[{_2F_{1}}(A+kI, B+kI; C+kI; x)\Big]{(B)_{k}}{(C)_{k}^{-1}}.
\label{2eqpp4}
\end{align}
This establishes (\ref{2eq3}) for $s= t+1$.  Hence result (\ref{2eq3}) is true for all values of $s$. The second recursion formula (\ref{2eq4}) is proved in a similar manner.
\end{proof} 

The recursion formulas for ${_2F_{1}}(A, B\pm sI; C; x)$  are obtained by replacing $A \leftrightarrow B$  in Theorems~\ref{rth12} --~\ref{rth21}.\\

We now present recursion formulas for the Gauss hypergeometric matrix function  $_2F_{1}(A, B; C; x)$ about the matrix $C$. 
\begin{theorem}
Let $C-sI$ be an invertible matrix  and let  $B$, $C$ be commuting matrices.  Then the following recursion formula holds true for the Gauss hypergeometric matrix function $_2F_{1}(A, B; C; x)$:
\begin{align}
&_2F_{1}(A, B; C-sI; x)\notag\\
&={_2F_{1}}(A, B; C; x)\notag\\
&\quad + x A\sum_{k=1}^{s}
 \,\Big[{_2F_{1}}(A+I, B+I; C+(2-k)I; x) \Big]B{(C-kI)^{-1}(C-(k-1)I)^{-1}}.
\label{2eq9}
\end{align}
\end{theorem}
\begin{proof}  
From the definition of the Gauss hypergeometric matrix function  $_2F_{1}$ and the transformation 
\begin{align}
(C-I)^{-1}_{n}=(C)^{-1}_{n}\left[I+{n}{(C-I)^{-1}}\right],
\end{align}
we can easily get the contiguous relation
\begin{align}
&_2F_{1}(A, B; C-I; x)\notag\\
&={_2F_{1}}(A, B; C; x)+ x A \,\Big[{_2F_{1}}(A+I, B+I; C+I; x)\Big]B{C^{-1}(C-I)^{-1}}.
\label{2eqp9}
\end{align}
Applying this contiguous relation twice on the Gauss hypergeometric matrix function  $_2F_{1}$ with matrix $C-2I$, we have
\begin{align}
&_2F_{1}(A, B; C-2I; x)\notag\\
&={_2F_{1}}(A, B; C-I; x)+ x A \,\Big[{_2F_{1}}(A+I, B+I; C; x)\Big]B{(C-I)^{-1}(C-2I)^{-1}}\notag\\
&={_2F_{1}}(A, B; C; x)+ x A \Big[{{_2F_{1}}(A+I, B+I; C+I; x)}B{C^{-1}(C-I)^{-1}}\notag\\
&\quad +\,{{_2F_{1}}(A+I, B+I; C; x)}B{(C-I)^{-1}(C-2I)^{-1}}\Big].
\label{2eqp19}
\end{align}
Iterating this method $s$ times on the  Gauss hypergeometric matrix function  $_2F_{1}(A, B; C-sI; x)$, we get the recursion formula (\ref{2eq9}).
\end{proof} 

\section{Recursion formulas for the Appell  matrix functions}
In this section, we obtain recursion formulas  for the  Appell  matrix functions. 
\begin{theorem}
Let $A+sI$ be an invertible matrix for all  integers $s\geq0$ and let $AB=BA$ and  $B'C=CB'$ then the following recursion formula holds true for the  Appell matrix function $F_{1}$:
\begin{align}
&F_{1}(A+sI, B, B'; C; x, y)\notag\\
&=F_{1}(A, B, B'; C; x, y)+ x{B}\left[\sum_{k=1}^{s}F_{1}(A+kI, B+I, B'; C+I; x, y)\right]{C}^{-1}\notag\\
&\quad +y\left[\sum_{k=1}^{s}F_{1}(A+kI, B, B'+I; C+I; x, y)\right]{B'}{C}^{-1}.
\label{3eq1}
\end{align}
Furthermore, if $A-kI$ is invertible for  integers $k\leq s$, then
\begin{align}
&F_{1}(A-sI, B, B'; C; x, y)\notag\\
&=F_{1}(A, B, B'; C; x, y)- x{B}\left[\sum_{k=0}^{s-1}F_{1}(A-kI, B+I, B'; C+I; x, y)\right]{C}^{-1}\notag\\
&\quad -y\left[\sum_{k=0}^{s-1}F_{1}(A-kI, B, B'+I; C+I; x, y)\right]{B'}{C}^{-1}.\label{3eq2}
\end{align}
\end{theorem}
\begin{proof}  Using the definition of  the Appell matrix function $F_{1}$ and the transformation
\begin{align}\notag(A+I)_{m+n}= A^{-1}(A)_{m+n}\left(A+{mI}+{nI}\right)\end{align}
we obtain the following contiguous relation:
\begin{align}
&F_{1}(A+I, B, B'; C; x, y)\notag\\
&=F_{1}(A, B, B'; C; x, y)+ x{B}\Big[F_{1}(A+I, B+I, B'; C+I; x, y)\Big]{C}^{-1}\notag\\
&\quad +y\Big[F_{1}(A+I, B, B'+I; C+I; x, y)\Big]{B'}{C}^{-1}.
\label{3eqp1}
\end{align}                                                             
To calculate contiguous  relation for  $F_{1}(A+2I, B, B'; C; x, y)$,  we replace $A$ with $A+I$ in (\ref{3eqp1}) and use (\ref{3eqp1}). This gives 
\begin{align}
&F_{1}(A+2I, B, B'; C; x, y)\notag\\
&=F_{1}(A, B, B'; C; x, y)\notag\\
&\quad + x{B}\Big[F_{1}(A+I, B+I, B'; C+I; x, y) +F_{1}(A+2I, B+I, B'; C+I; x, y)\Big] {C}^{-1}\notag\\
&\quad+y\Big[F_{1}(A+I, B, B'+I; C+I; x, y) 
+F_{1}(A+2I, B, B'+I; C+I; x, y)\Big]{B'}{C}^{-1}.
\label{3eqp2}
\end{align}  
Iterating this process $s$ times,  we obtain (\ref{3eq1}). For  the proof of (\ref{3eq2}),  replace the matrix  $A$ with $A-I$ in (\ref{3eqp1}). As $A-I$ is invertible, this gives 
\begin{align}
&F_{1}(A-I, B, B'; C; x, y)\notag\\
&=F_{1}(A, B, B'; C; x, y)- x{B}\Big[F_{1}(A, B+I, B'; C+I; x, y)\Big]{C}^{-1}\notag\\
&\quad -y\Big[F_{1}(A, B, B'+I; C+I; x, y)\Big]{B'}{C}^{-1}.
\label{3eqp21}
\end{align}
Iteratively, we get (\ref{3eq2}).
\end{proof} 

From the contiguous relations (\ref{3eqp1})   and  (\ref{3eqp21}), we can write the recursion formulas for Appell  matrix function $F_{1}$ in the following forms.
\begin{theorem}
Let $A+sI$ be an invertible matrix for all  integers $s\geq0$ and  let $AB=BA$ and  $B'C=CB'$ then the the following recursion formula holds true for the  Appell matrix function $F_{1}$:
\begin{align}
&F_{1}(A+sI, B, B'; C; x, y)\notag\\
&=\sum_{k_1+k_2\leq s}^{}{s\choose k_1, k_2}(B)_{k_1}\, x^{k_1} y^{k_2}\,\notag\\
&\quad\times\Big[{F_{1}}(A+(k_1+k_2)I, B+k_1I, B'+k_2I; C+(k_1+k_2)I; x, y)\Big](B')_{k_2}{(C)^{-1}_{k_1+k_2}};
\label{3eq3}
\end{align}
Furthermore, if $A-kI$ is invertible for  integers $k\leq s$, then 
\begin{align}
&F_{1}(A-sI, B, B'; C; x, y)\notag\\
&=\sum_{k_1+k_2\leq s}^{}{s\choose k_1, k_2}(B)_{k_1}\, (-x)^{k_1} (-y)^{k_2}\,\notag\\
&\quad\times\Big[{F_{1}}(A, B+k_1I, B'+k_2I; C+(k_1+k_2)I; x, y)\Big](B')_{k_2}{(C)^{-1}_{k_1+k_2}},
\label{3eq4}
\end{align}
where ${s\choose k_1, k_2}=\frac{s!}{k_1!k_2!(s-k_1-k_2)!}$.
\end{theorem}
\begin{proof} The proof of (\ref{3eq3}) is based upon the principle of mathematical induction on  $s\in\mathbb{N}$. For $s=1$, the result (\ref{3eq3}) is true obviously. Suppose  (\ref{3eq3}) is true for $s=t$, that is,
 \begin{align}
&F_{1}(A+tI, B, B'; C; x, y)\notag\\
&=\sum_{k_1+k_2\leq t}^{}{t\choose k_1, k_2}(B)_{k_1}\, x^{k_1} y^{k_2}\,\notag\\
 &\quad\times{F_{1}}(A+(k_1+k_2)I, B+k_1I, B'+k_2I; C+(k_1+k_2)I; x, y)(B')_{k_2}{(C)^{-1}_{k_1+k_2}},
 \label{3eqp3}
 \end{align}
 Replacing $A$ with $A+I$ in (\ref{3eqp3}) and using the contiguous relation (\ref{3eqp1}), we get
 \begin{align}
 &F_{1}(A+tI+I, B, B'; C; x, y)\notag\\
 &=\sum_{k_1+k_2\leq t}^{}{t\choose k_1, k_2}(B)_{k_1}\, x^{k_1} y^{k_2}\,\notag\\
 &\quad\times\Big[{F_{1}}(A+(k_1+k_2)I, B+k_1I, B'+k_2I; C+(k_1+k_2)I; x, y)+x{(B+k_1I)}\notag\\
 &\quad\times{F_{1}}(A+(k_1+k_2)I+I, B+k_1I+I, B'+k_2I; C+(k_1+k_2)I+I; x, y){(C+(k_1+k_2)I)^{-1}}\notag\\
 &\quad\times+y {F_{1}}(A+(k_1+k_2)I+I, B+k_1I, B'+k_2I+I; C+(k_1+k_2)I+I; x, y)\notag\\ &\quad\times{(B'+k_2I)}{(C+(k_1+k_2)I)^{-1}}\Big](B')_{k_2}{(C)^{-1}_{k_1+k_2}}.
 \label{3eqp4}
 \end{align}
 Simplifying, (\ref{3eqp4}) takes the form
 \begin{align}
 &F_{1}(A+tI+I, B, B'; C; x, y)\notag\\
 &=\sum_{k_1+k_2\leq t}^{}{t\choose k_1, k_2}(B)_{k_1}\, x^{k_1} y^{k_2}\,\notag\\
 &\quad\times\Big[{F_{1}}(A+(k_1+k_2)I, B+k_1I, B'+k_2I; C+(k_1+k_2)I; x, y)\notag\\
 &\quad +\sum_{k_1+k_2\leq t+1}^{}{t\choose k_1-1, k_2}{F_{1}}(A+(k_1+k_2)I, B+k_1I, B'+k_2I; C+(k_1+k_2)I; x, y)\notag\\
 &\quad +\sum_{k_1+k_2\leq t+1}^{}{t\choose k_1, k_2-1}\notag\\&\times{F_{1}}(A+(k_1+k_2)I, B+k_1I, B'+k_2I; C+(k_1+k_2)I; x, y)\Big](B')_{k_2}{(C)^{-1}_{k_1+k_2}}.
 \label{3eqp5}
 \end{align}
 Using Pascal's identity in (\ref{3eqp5}), we have
 \begin{align}
& F_{1}(A+(t+1)I, B, B'; C; x, y)\notag\\
&=\sum_{k_1+k_2\leq t+1}^{}{t+1\choose k_1, k_2}(B)_{k_1}\, x^{k_1} y^{k_2}\,\notag\\
 &\quad\times{F_{1}}(A+(k_1+k_2)I, B+k_1I, B'+k_2I; C+(k_1+k_2)I; x, y)(B')_{k_2}{(C)^{-1}_{k_1+k_2}}.
 \end{align}
 This establishes (\ref{3eq3}) for $s=t+1$. Hence by induction, result given in (\ref{3eq3}) is true for all values of $s$. The second recursion formula (\ref{3eq4}) can be proved in a similar manner.
 \end{proof} 
 Now, we present the recursion formulas for the parameters $B$, $B'$ of the Appell matrix function $F_{1}$. We omit the proofs of the given below theorems.
\begin{theorem}
Let $B+sI$ and $B'+sI$ be invertible matrices for all  integers $s\geq0$. Then following recursion formulas hold true for the  Appell matrix function $F_{1}$:
\begin{align}
&F_{1}(A, B+sI, B'; C; x, y)\notag\\
&=F_{1}(A, B, B'; C; x, y)+ x{A}\Big[\sum_{k=1}^{s}F_{1}(A+I, B+kI, B'; C+I; x, y)\Big]{C}^{-1};
\label{3eq5}
\end{align}
\begin{align}
&F_{1}(A, B, B'+sI; C; x, y)\notag\\
&=F_{1}(A, B, B'; C; x, y)+ y{A}\Big[\sum_{k=1}^{s}F_{1}(A+I, B, B'+kI; C+I; x, y)\Big]{C}^{-1}.
\label{3eq7}
\end{align}
Furthermore, if $B-kI$ and $B'-kI$ are invertible for  integers $k\leq s$, then
\begin{align}
&F_{1}(A, B-sI, B'; C; x, y)\notag\\
&=F_{1}(A, B, B'; C; x, y)- x{A}\Big[\sum_{k=0}^{s-1}F_{1}(A+I, B-kI, B'; C+I; x, y)\Big]{C}^{-1};
\label{3eq6}
\end{align}
\begin{align}
&F_{1}(A, B, B'-sI; C; x, y)\notag\\
&=F_{1}(A, B, B'; C; x, y)- y{A}\Big[\sum_{k=0}^{s-1}F_{1}(A+I, B, B'-kI;  C+I; x, y)\Big]{C}^{-1}.\label{3eq8}
\end{align}
\end{theorem}
\begin{theorem}
Let $B+sI$ and $B'+sI$ be invertible matrices for all $s\geq0$. Then following recursion formulas hold true for the  Appell matrix function $F_{1}$:
\begin{align}
&F_{1}(A, B+sI, B'; C; x, y)\notag\\
&=\sum_{k_1=0}^{s}{s\choose k_1}{(A)_{k_1}}\, x^{k_1}   \Big[\ {F_{1}}(A+k_1I, B+k_1I, B'; C+k_1I; x, y)\Big]{(C)^{-1}_{k_1}};
\label{3eq9}
\end{align}
\begin{align}
&F_{1}(A, B, B'+sI; C; x, y)\notag\\
&=\sum_{k_1=0}^{s}{s\choose k_1}{(A)_{k_1}}\, y^{k_1} \Big[\ {F_{1}}(A+k_1I, B, B'+k_1 I;  C+k_1I; x, y)\Big]{(C)^{-1}_{k_1}}.
\label{3eq11}
\end{align}
Furthermore, if $B-kI$ and $B'-kI$ are invertible for  integers $k\leq s$, then
\begin{align}
&F_{1}(A, B-sI, B'; C; x, y)\notag\\
&=\sum_{k_1=0}^{s}{s\choose k_1}{(A)_{k_1}}\, (-x)^{k_1} \Big[{F_{1}}(A+k_1I, B, B'; C+k_1I; x, y)\Big]{(C)^{-1}_{k_1}};
\label{3eq10}
\end{align}
\begin{align}
&F_{1}(A, B, B'-sI; C; x, y)\notag\\
&=\sum_{k_1=0}^{s}{s\choose k_1}{(A)_{k_1}}\, (-y)^{k_1} \Big[{F_{1}}(A+k_1I, B, B';  C+k_1I; x, y)\Big]{(C)^{-1}_{k_1}}.
\label{3eq12}
\end{align}
\end{theorem}
\begin{theorem}
Let $C-sI$ be an invertible matrix for all  integers $s\geq0$ and let $AB=BA$,  $B'C=CB'$ . Then the following recursion formula holds true for the  Appell matrix function $F_{1}$:
\begin{align}
&F_{1}(A, B, B'; C-sI; x, y)\notag\\
&=F_{1}(A, B, B'; C; x, y)+ x AB\sum_{k=1}^{s}\Big[ F_{1}(A+I, B+I, B'; C+(2-k)I; x, y)\Big]\notag\\
&\times{(C-kI)^{-1}(C-(k-1)I)^{-1}}+ y A\sum_{k=1}^{s}\Big[ F_{1}(A+I, B, B'+I; C+(2-k)I; x, y)\Big]\notag\\
&\times B'{(C-kI)^{-1}(C-(k-1)I)^{-1}}.\label{3eq13}
\end{align}
\end{theorem}
\begin{proof} Applying the definition of the Appell matrix function $F_{1}$ and the relation
\begin{align*}
{(C-I)^{-1}_{m+n}}={(C)^{-1}_{m+n}}\left[1+{m}{(C-I)^{-1}}+{n}{(C-I)^{-1}}\right],
\end{align*} \\
we obtain the following contiguous relation:
\begin{align}
&F_{1}(A, B, B'; C-I; x, y)\notag\\
&=F_{1}(A, B, B'; C; x, y)+ {x AB}\Big[\,F_{1}(A+I, B+I, B'; C+I; x, y)\Big]{C^{-1}(C-I)^{-1}}\notag\\
&\quad + y A\Big[\,F_{1}(A+I, B, B'+I; C+I; x, y)\Big]B'{C^{-1}(C-I)^{-1}}.\label{3eqp13}
\end{align}
Using this contiguous relation to the Appell matrix function $F_{1}$ with the matrix $C-sI$ for $s$ times, we get (\ref{3eq13}). 
\end{proof} 

Next, we will present recursion formulas for the Appell matrix functions $F_{2}$, $F_{3}$, $F_{4}$. We omit the proofs of the given below theorems.
\begin{theorem}
Let $A+sI$ be an invertible matrix for all  integers $s\geq0$ and let $AB=BA$;  $B'$, $C$ and $C'$ be commuting matrices. Then the following recursion formula holds true for the  Appell matrix function $F_{2}$:
\begin{align}
&F_{2}(A+sI, B, B'; C, C'; x, y)\notag\\
&=F_{2}(A, B, B'; C, C'; x, y)+ x{B}\Big[\sum_{k=1}^{s}F_{2}(A+kI, B+I, B'; C+I, C'; x, y)\Big]{C}^{-1}\notag\\
&\quad +y\Big[\sum_{k=1}^{s}F_{2}(A+kI, B, B'+I;C, C'+I; x, y)\Big]{B'}{C'}^{-1}.\label{3eq14}
\end{align}
If $A-kI$ is invertible for  integers $k\leq s$, then 
\begin{align}
&F_{2}(A-sI, B, B'; C, C'; x, y)\notag\\
&=F_{2}(A, B, B'; C, C'; x, y)- x{B}\Big[\sum_{k=0}^{s-1}F_{2}(A-kI, B+I, B'; C+I, C'; x, y)\Big]{C}^{-1}\notag\\
&\quad -y\Big[\sum_{k=0}^{s-1}F_{2}(A-kI, B, B'+I;C, C'+I; x, y)\Big]{B'}{C'}^{-1}.\label{3eq15}
\end{align}
\end{theorem}
\begin{theorem}
Let $A+sI$ be an invertible matrix for all  integers $s\geq0$ and let $AB=BA$; $B'$, $C$ and $C'$ be commuting matrices. Then the following recursion formula holds true for the  Appell matrix function $F_{2}$:
\begin{align}
&F_{2}(A+sI, B, B'; C, C'; x, y)\notag\\
&=\sum_{k_1+k_2\leq s}^{}{s\choose k_1, k_2}(B)_{k_1}\, x^{k_1} y^{k_2}\,
\Big[{F_{2}}(A+(k_1+k_2)I, B+k_1I, B'+k_2I; C+k_1I, C'+k_2 I; x, y)\Big]\notag\\
&\quad\times(B')_{k_2}{(C)^{-1}_{k_1} (C')^{-1}_{k_2}}.\notag\\
\label{3eq16}
\end{align}
If $A-kI$ is invertible for  integers  $k\leq s$, then
\begin{align}
&F_{2}(A-sI, B, B'; C, C'; x, y)\notag\\
&=\sum_{k_1+k_2\leq s}^{}{s\choose k_1, k_2}(B)_{k_1}\, (-x)^{k_1} (-y)^{k_2}\,\Big[{F_{2}}(A, B+k_1I, B'+k_2I; C+k_1I, C'+k_2 I; x, y)\Big]\notag\\
&\quad\times(B')_{k_2}{(C)^{-1}_{k_1}(C')^{-1}_{k_2}}.
\label{3eq17}
\end{align}
\end{theorem}
\begin{theorem}
Let $B+sI$ and $B'+sI$ be invertible matrices for all  integers  $s\geq0$ and let  $CC'=C'C$ . Then following recursion formulas hold true for the  Appell matrix function $F_{2}$:
\begin{align}
&F_{2}(A, B+sI, B'; C, C'; x, y)\notag\\
&=F_{2}(A, B, B'; C, C'; x, y)+ x{A}\Big[\sum_{k=1}^{s}F_{2}(A+I, B+kI, B'; C+I, C'; x, y)\Big]{C}^{-1};\label{3eq18}\\
&F_{2}(A, B, B'+sI; C, C'; x, y)\notag\\
&=F_{2}(A, B, B'; C, C'; x, y)+ y{A}\Big[\sum_{k=1}^{s}F_{2}(A+I, B, B'+kI; C, C'+I; x, y)\Big]{C'}^{-1}.\label{3eq20}
\end{align}
Furthermore, if $B-kI$ and $B'-kI$ are invertible for  integers $k\leq s$, then 
\begin{align}
&F_{2}(A, B-sI, B'; C, C'; x, y)\notag\\
&=F_{2}(A, B, B'; C, C'; x, y)- x{A}\Big[\sum_{k=0}^{s-1}F_{2}(A+I, B-kI, B'; C+I, C'; x, y)\Big]{C}^{-1};\label{3eq19}\\
&F_{2}(A, B, B'-sI; C, C'; x, y)\notag\\
&=F_{2}(A, B, B'; C, C'; x, y)- y{A}\Big[\sum_{k=0}^{s-1}F_{2}(A+I, B, B'-kI; C, C'+I; x, y)\Big]{C'}^{-1}.\label{3eq21}
\end{align}
\end{theorem}
\begin{theorem}
Let $B+sI$ and $B'+sI$ be invertible matrices for all  integers $s\geq0$ and let  $CC'=C'C$ . Then following recursion formulas hold true for the  Appell matrix function $F_{2}$:
\begin{align}
&F_{2}(A, B+sI, B'; C, C'; x, y)\notag\\
&=\sum_{k_1=0}^{s}{s\choose k_1}{(A)_{k_1}}\, x^{k_1} \Big[ {F_{2}}(A+k_1I, B+k_1I, B'; C+k_1I, C'; x, y)\Big]{(C)^{-1}_{k_1}};
\label{3eq22}\\
&F_{2}(A, B, B'+sI; C, C'; x, y)\notag\\
&=\sum_{k_1=0}^{s}{s\choose k_1}{(A)_{k_1}}\, y^{k_1} \Big[ {F_{2}}(A+k_1I, B, B'+k_1 I;C,  C'+k_1I; x, y)\Big]{(C')^{-1}_{k_1}}.
\label{3eq24}
\end{align}
Furthermore, if $B-kI$ and $B'-kI$ are invertible for  integers $k\leq s$, then
\begin{align}
&F_{2}(A, B-sI, B'; C, C'; x, y)\notag\\
&=\sum_{k_1=0}^{s}{s\choose k_1}{(A)_{k_1}}\, (-x)^{k_1} \Big[ {F_{2}}(A+k_1I, B, B'; C+k_1I, C'; x, y)\Big]{(C)^{-1}_{k_1}};
\label{3eq23}\\
&F_{2}(A, B, B'-sI; C, C'; x, y)\notag\\
&=\sum_{k_1=0}^{s}{s\choose k_1}{(A)_{k_1}}\, (-y)^{k_1} \Big[ {F_{2}}(A+k_1I, B, B'; C, C'+k_1I; x, y)\Big]{(C')^{-1}_{k_1}}.
\label{3eq25}
\end{align}
\end{theorem}
\begin{theorem}
Let $C-sI$ and $C'-sI$ be invertible matrices for all  integers $s\geq0$ and let $AB=BA$; $B'$, $C$ and $C'$ be commuting matrices. Then following recursion formulas hold true for the  Appell matrix function $F_{2}$:
\begin{align}
&F_{2}(A, B, B'; C-sI, C'; x, y)\notag\\
&=F_{2}(A, B, B'; C, C'; x, y)+ x AB\Big[\sum_{k=1}^{s} F_{2}(A+I, B+I, B'; C+(2-k)I, C'; x, y)\notag\\&\quad\times{(C-kI)^{-1}(C-(k-1)I)^{-1}}\Big]\,
;\label{3eq26}
\end{align}
\begin{align}
&F_{2}(A, B, B'; C, C'-sI; x, y)\notag\\
&=F_{2}(A, B, B'; C, C'; x, y)+ y A\Big[\sum_{k=1}^{s} F_{2}(A+I, B, B'+I;\,C, C'+(2-k)I; x, y)\,\notag\\
&\quad\times{(C'-kI)^{-1}(C'-(k-1)I)^{-1}}\Big]B'
.\label{3eq27}
\end{align}
\end{theorem}

\begin{theorem}\label{rth1}
Let $A+sI$ and $A'+sI$ be invertible matrices for all  integers $s\geq0$ and let $A$, $A'$  and $B$ be commuting matrices;  $B'C=CB'$. Then following recursion formulas hold true for the  Appell matrix function $F_{3}$:
\begin{align}
&F_{3}(A+sI, A', B, B'; C; x, y)\notag\\
&=F_{3}(A, A', B, B'; C; x, y)+ x{B}\Big[\sum_{k=1}^{s}F_{3}(A+kI, A', B+I, B'; C+I; x, y)\Big]{C}^{-1};\label{3eq30}\\
&F_{3}(A, A'+sI, B, B'; C; x, y)\notag\\
&=F_{3}(A, A', B, B'; C; x, y)+ y\Big[\sum_{k=1}^{s}F_{3}(A, A'+kI, B, B'+I; C+I; x, y)\Big]{B'}{C}^{-1}.\label{3eq32}
\end{align}
Furthermore, if $A-kI$ and $A'-kI$ are invertible for  integers $k\leq s$, then 
\begin{align}
&F_{3}(A-sI, A', B, B'; C; x, y)\notag\\
&=F_{3}(A, A', B, B'; C; x, y)-x{B}\Big[\sum_{k=0}^{s-1}F_{3}(A-kI, A', B+I, B'; C+I; x, y)\Big]{C}^{-1};\label{3eq31}\\
&F_{3}(A, A'-sI, B, B'; C; x, y)\notag\\
&=F_{3}(A, A', B, B'; C; x, y)-y\Big[\sum_{k=0}^{s-1}F_{3}(A, A'-kI, B, B'+I; C+I; x, y)\Big]{B'}{C}^{-1};\label{3eq33}
\end{align}
\end{theorem}
\begin{theorem}\label{rth2}
Let $A+sI$ and $A'+sI$ be invertible matrices for all  integers $s\geq0$ and  let $A$, $A'$  and $B$ be commuting matrices;  $B'C=CB'$.  Then following recursion formulas hold true for the  Appell matrix function $F_{3}$:
\begin{align}
&F_{3}(A+sI,A', B, B'; C; x, y)\notag\\
&=\sum_{k_1=0}^{s}{s\choose k_1}{(B)_{k_1}}\, x^{k_1} \Big[{F_{3}}(A+k_1 I, A', B+k_1I, B'; C+k_1 I; x, y)\Big]{(C)^{-1}_{k_1}};
\label{3eq34}\\
&F_{3}(A,A'+sI, B, B'; C; x, y)\notag\\
&=\sum_{k_1=0}^{s}{s\choose k_1}\, y^{k_1} \Big[{F_{3}}(A, A'+k_1 I, B, B'+k_1I; C+k_1 I; x, y)\Big]{(B')_{k_1}}{(C)^{-1}_{k_1}}.
\label{3eq36}
\end{align}
Furthermore, if $A-kI$ and $A'-kI$ are invertible for integers  $k\leq s$, then 
\begin{align}
&F_{3}(A-sI, A', B, B'; C; x, y)\notag\\
&=\sum_{k_1=0}^{s}{s\choose k_1}{(B)_{k_1}}\, (-x)^{k_1} \Big[{F_{3}}(A, A', B+k_1I, B'; C+k_1 I; x, y)\Big]{(C)^{-1}_{k_1}};
\label{3eq35}\\
&F_{3}(A, A'-sI, B, B'; C; x, y)\notag\\
&=\sum_{k_1=0}^{s}{s\choose k_1}\, (-y)^{k_1} \Big[ {F_{3}}(A, A', B, B'+k_1I; C+k_1 I; x, y)\Big]{(B')_{k_1}}{(C)^{-1}_{k_1}}.
\label{3eq37}
\end{align}
\end{theorem}
\bigskip
The recursion formulas for $F_{3}(A, A', B\pm sI, B'; C; x, y)$ and $F_{3}(A, A', B,  B'\pm sI; C; x, y)$ are obtained by replacing $A \leftrightarrow B$ and $A' \leftrightarrow B'$ in Theorem~\ref{rth1} and Theorem~\ref{rth2}, respectively.

\begin{theorem}
Let $C-sI$ be an invertible matrix for all  integers $s\geq0$ and  let $A$, $A'$  and $B$ be commuting matrices;  $B'C=CB'$. Then the following recursion formula holds true for the  Appell matrix function $F_{3}$:
\begin{align}
&F_{3}(A,A', B, B'; C-sI; x, y)\notag\\
&=F_{3}(A,A', B, B'; C; x, y)+ x AB\Big[\sum_{k=1}^{s}\, F_{3}(A+I, A', B+I, B'; C+(2-k)I; x, y)\Big]\notag\\&\quad\times{(C-kI)^{-1}(C-(k-1)I)^{-1}}\notag\\
&\quad + yA' \Big[\sum_{k=1}^{s}\,F_{3}(A, A'+I, B, B'+I; C+(2-k)I; x, y){(C-kI)^{-1}(C-(k-1)I)^{-1}}\Big] B'.\label{3eq38}
\end{align}
\end{theorem}

\begin{theorem}\label{rth3}
Let $A+sI$ be an invertible matrix for all  integers  $s\geq0$ and let $AB=BA$; $CC'=C'C$. Then the following recursion formula holds true for the  Appell matrix function $F_{4}$:
\begin{align}
&F_{4}(A+sI, B; C, C'; x, y)\notag\\
&=F_{4}(A, B; C, C'; x, y)+ x{B}\Big[\sum_{k=1}^{s}F_{4}(A+kI, B+I; C+I, C'; x, y)\Big]{C}^{-1}\notag\\
&\quad +y{B}\Big[\sum_{k=1}^{s}F_{4}(A+kI, B+I;C,  C'+I; x, y)\Big]{C'}^{-1}.\label{3eq39}\end{align}
Furthermore, if $A-kI$ is invertible for $k\leq s$, then
\begin{align}
&F_{4}(A-sI, B; C, C'; x, y)\notag\\
&=F_{4}(A, B; C, C'; x, y)- x{B}\Big[\sum_{k=0}^{s-1}F_{4}(A-kI, B+I; C+I, C'; x, y)\Big]{C}^{-1}\notag\\
&\quad -y{B}\Big[\sum_{k=0}^{s-1}F_{1}(A-kI, B+I; C,  C'+I; x, y)\Big]{C'}^{-1}.\label{3eq40}
\end{align}
\end{theorem}
\begin{theorem}\label{rth4}
Let $A+sI$ be an invertible matrix for all  integers $s\geq0$ and let $AB=BA$; $CC'=C'C$. Then the following recursion formula holds true for the  Appell matrix function $F_{4}$:
\begin{align}
&F_{4}(A+sI, B; C, C'; x, y)\notag\\
&=\sum_{k_1+k_2\leq s}^{}{s\choose k_1, k_2}{(B)_{k_1+k_2}}\, x^{k_1} y^{k_2}\,\notag\\
&\quad\times\Big[{F_{4}}(A+(k_1+k_2)I, B+(k_1+k_2)I; C+k_1 I, C'+k_2 I; x, y)\Big]{(C)^{-1}_{k_1} (C')^{-1}_{k_2}}.
\label{3eq41}
\end{align}
Furthermore, if $A-kI$ is invertible for  integers  $k\leq s$, then
\begin{align}
&F_{4}(A-sI, B;C, C'; x, y)\notag\\
&=\sum_{k_1+k_2\leq s}^{}{s\choose k_1, k_2}{(B)_{k_1+k_2}}\, (-x)^{k_1} (-y)^{k_2}\,\notag\\
&\quad\times\Big[{F_{4}}(A, B+(k_1+k_2)I; C+k_1 I, C'+k_2 I; x, y)\Big]{(C)^{-1}_{k_1} (C')^{-1}_{k_2}}.
\label{3eq42}
\end{align}
\end{theorem}\bigskip
The recursion formulas for $F_{4}(A, B\pm sI; C, C'; x, y)$  are obtained by replacing $A \leftrightarrow B$  in Theorems~\ref{rth3} --~\ref{rth4}.

\begin{theorem}
Let $C-sI$ and $C'-sI$ be invertible matrices for all   integers $s\geq0$ and  let $AB=BA$; $CC'=C'C$. Then following recursion formulas hold true for the  Appell matrix function $F_{4}$:
\begin{align}
&F_{4}(A, B; C-sI, C'; x, y)\notag\\
&=F_{4}(A, B; C, C'; x, y)\notag\\&+ x AB\Big[\sum_{k=1}^{s}\, F_{4}(A+I, B+I; C+(2-k)I, C'; x, y){(C-kI)^{-1}(C-(k-1)I)^{-1}}\Big];\label{3eq43}
\end{align}
\begin{align}
&F_{4}(A, B; C, C'-sI; x, y)\notag\\
&=F_{4}(A, B; C, C'; x, y)\notag\\&+ y AB\Big[\sum_{k=1}^{s}\, F_{4}(A+I, B+I;\,C, C'+(2-k)I; x, y){(C'-kI)^{-1}(C'-(k-1)I)^{-1}}\Big].\label{3eq44}
\end{align}
\end{theorem}

\end{document}